\documentclass[ reqno, 11pt]{amsart}

\usepackage[text={6in,8.5in},centering,letterpaper,dvips]{geometry}
\usepackage{amsmath, amscd,amsthm}%
\usepackage{amsfonts}%
\usepackage{amssymb}%
\usepackage{graphicx}
\usepackage[normalem]{ulem} 
\usepackage{diagrams}
\usepackage{newclude} 
\usepackage{enumerate}
\usepackage{multido} 
\usepackage{hyperref}

\newtheorem{theorem}{Theorem}
\theoremstyle{plain}

\newtheorem{corollary}{Corollary}

\newtheorem{lemma}{Lemma}

\newtheorem{proposition}{Proposition}

\numberwithin{equation}{section}

\theoremstyle{definition}
\newtheorem{definition}{Definition}
\newtheorem{example}{Example}
\newtheorem*{acknowledgement}{Acknowledgement}

\theoremstyle{remark}
\newtheorem*{remark}{Remark}
\newtheorem*{note}{Note}

\newcommand{\cindex}[1]{\ensuremath{\mathcal{I}(#1)} }

\newcommand{\Inv}[1]{\ensuremath{\operatorname{Inv}(#1)}}

\newcommand{\Inter}[1]{\ensuremath{\operatorname{Int}(#1)}}

\newcommand{\Spec}[1]{\ensuremath{\operatorname{Spec}(#1)}}

\newcommand{\exhaust}{asymptotically-invariant exhausting }
\newcommand{\goodvf}{permissible }

\begin{document}

\title{On the stable Conley index in Hilbert spaces}     
\author{Tirasan Khandhawit}              
\date{} 
\address{Kavli Institute for the Physics and Mathematics of the Universe (WPI), Todai Institutes for Advanced
Study, The University of Tokyo, 5-1-5 Kashiwa-No-Ha, Kashiwa, Chiba 277-8583, Japan}
\email{tirasan.khandhawit@ipmu.jp}
\urladdr{} 

\begin{abstract}
In this paper, we study Conley theory in Hilbert spaces and make some refinement of the construction of the stable Conley index developed by G\c{e}ba, Izydorek, and Pruszko. For instance, we allow subspaces other than invariant subspaces in the the construction. As a main result, we show that the resulting stable Conley index for a given  isolating neighborhood of a flow does not depend on choices in the construction. We also hope that some results can be applied to the construction of Floer homotopy type.  
\end{abstract}

\maketitle


\section{Introduction}
Conley index theory concerns about some homotopy invariant of a flow and its isolated invariant set. One might view Conley index theory as a generalization of Morse theory.
However, Conley index theory  was developed for a flow on a locally compact space, and one cannot apply it directly in an infinite-dimensional setting. In 1999, G\c{e}ba, Izydorek, and Pruszko \cite{Geba} developed a  generalization of Conley index theory for a special class of flows on a  Hilbert space. In particular, they associated  the stable Conley index, as a stable equivalence class of spectra, to an isolating neighborhood of such a flow. 

The main idea of the construction is finite dimensional approximation of the Conley index. More specifically, one considers a flow compressed on sufficiently large finite-dimensional subspaces and tries to obtain the Conley index for each of these subspaces. A collection of these Conley indices then forms an object in a stable homotopy category. 

In this paper, we will consider a flow  which is generated by a vector field $F$ with a special form. For example, we suppose that there is a decomposition $F = L + Q$ where $L$ is bounded linear and $Q$ is (possibly nonlinear) compact. When $V$ is an $L$-invariant subspace, we can look at a \textit{compressed vector field} $L + \pi_V Q$, where $\pi_V$ is the orthogonal projection onto $V$. 

One of the goals of this paper is to include subspaces which are not necessarily \(L\)-invariant in the construction. In particular, we characterize a class of subspaces which are suitable for finite dimensional approximation (see Proposition~\ref{prop indep}). As a main result, we show that the resulting stable Conley index does not depend on choices in the construction such as the choice of decompositions of $F$ and the choice of subspaces. The following result is a direct consequence of Corollary~\ref{cor invariance}. 

\begin{theorem} Let $X$ be an isolating neighborhood with respect to a flow generated by $F$ on a Hilbert space. Suppose that there is a permissible decomposition $F = L + Q$ and that $L$ admits an \exhaust sequence.

Then, we can construct a stable Conley index $E(X,F,L)$ which is a spectrum and is well-defined up to stable homotopy equivalence. Different choices of $L$ only change $E(X,F,L)$ by suspension of spheres.
\end{theorem}

The main application of the stable Conley index is in the study of Hamiltonian systems as in \cite{Geba}.  We mention here that there is another construction of the stable Conley index in the context of gauge theory. In 2003, Manolescu \cite{Man1} also used the idea of finite dimensional approximation of the Conley index to construct a stable homotopy type of the Seiberg-Witten Floer homology for rational homology 3-spheres (see some subsequent works in \cite{TK1}, \cite{Man13-1}, and \cite{ManK}). We point out some differences in the setup here.
\begin{itemize}
\item A vector field in the Seiberg-Witten theory is of the form $L + Q$ where $L$ is a self-adjoint elliptic differential operator and $Q$ is a quadratic map.
\item In gauge theory, one often needs to work on the Sobolev completion of a space of sections. As a result, a flow is only partially defined because the vector field maps the $L^2_k$ completion to the $L^2_{k-1}$ completion, not to itself. Nevertheless, one can describe a trajectory as a solution of the equation $-\frac{d}{dt}(\eta(t)) = (L+Q)(\eta(t))$.
\item One needs to be more careful about the class of subspaces and projections because several Sobolev completions are involved. For example, we might need to consider only a subspace of smooth section so that the corresponding projection is a smoothing operator and orthogonal with respect to some $L^2_k$ inner product.
\end{itemize}

Despite differences in technical arguments between the two contexts, some results have similar features. For example, the fact that an infinite-dimensional isolating neighborhood is also an isolating neighborhood for compressed flows on sufficiently large subspaces holds in both context (compare Proposition~\ref{prop indep} with \cite[Proposition~3]{Man1} and \cite[Proposition~11]{TK1}). We hope that other results (specially in Section~\ref{sec stable}) of this paper  can be applied to the Seiberg-Witten case. 


\begin{acknowledgement}  This was part of the author's Ph.D. thesis at Massachusetts Institute of Technology. The author would like to gratefully thank Tom Mrowka for advising and support during the graduate study. This work was supported by World Premier International Research Center Initiative (WPI), MEXT, Japan.
\end{acknowledgement} 

\section{Classical Conley index theory} \label{sec conley}

We give some background in Conley index theory. A thorough treatment can be found in \cite{Conley} and \cite{SalaCon}. Let $\Omega$ be a locally compact, Hausdorff topological space. A \emph{flow} on $\Omega$ is a continuous map $\eta : \Omega \times \mathbb{R} \rightarrow \Omega$ such that $\eta(x,0) = x$ and $\eta(x, s + t) = \eta(\eta(x,s),t)$ for all $x \in \Omega$ and $s,t \in \mathbb{R}$. We will denote the image $\eta(x,t)$ by  $x \cdot t$ when the flow is understood.
 
\begin{definition} \
\begin{enumerate}
        \item The \emph{maximal invariant subset} \(\Inv{X,\eta}\) of a subset $X \subset \Omega$ is a set of points whose flow line stays in \(X\) for all time, i.e. $\Inv{X,\eta} = \left\{ x \in X | x \cdot \mathbb{R} \subseteq X \right\}$.
        \item A compact subset $X$ of $\Omega$ is called an \emph{isolating neighborhood}  if its maximal invariant subset $\Inv{X,\eta}$ is contained in its interior $\Inter{X}$.
                 \item A compact subset $S$ of $\Omega$ is called an \emph{isolated invariant set} if there is an isolating neighborhood $X$ so that $\Inv{X,\eta} = S$.
\end{enumerate}
\end{definition}
     
The main feature of Conley index theory is that one can associate some topological invariants to an isolated invariant set (or an isolating neighborhood) together with the flow. To proceed, we introduce the important concept of an index pair.

\begin{definition} \label{indexp1}
Let $S$ be an isolated invariant set. A pair of compact subsets $(N,L)$ is called an index pair for $S$ if the following conditions hold
\begin{enumerate}
        \item $S \subset N \backslash L $ (this implies $S = \Inv{N \backslash L , \eta}$),
        \item $L$ is positively invariant relative to $N$ (i.e. if  $x \in L$ and $x \cdot [0,t] \subset N$, then $x \cdot [0,t] \subset L$),
        \item $L$ is an exit set for $N$ (i.e. if $x \in N $ but $x \cdot [0, \infty)  \nsubseteq N$, then there exists $t>0$ such that $x \cdot [0,t] \subset N$ and $x \cdot t \in L$).
\end{enumerate}
\end{definition}

For an isolating neighborhood $X$ with $\Inv{X,\eta} = S$, we will also called $(N,L)$ an index pair for $X$ if it is an index pair for $S$. This definition does not depend on $X$ but it is sometimes more convenient to emphasize the isolating neighborhood instead of the isolated invariant set. 

Two fundamental results in Conley theory state that, given a fixed isolated invariant set (or a fixed isolating neighborhood), an index pair always exists and that all the pointed spaces of the form $(N/L , [L])$, where $(N,L)$ is an index pair, are homotopy equivalent. This leads to the following definition.
\begin{definition} For an isolated invariant set $S$, we define its (homotopy) Conley index $\cindex{S , \eta}$  to be a homotopy type of a pointed space $ (N/L , [L])$ where $(N,L)$ is an index pair for  $S$. For an isolating neighborhood \(X\) with $\Inv{X,\eta} = S$, we also define \(\cindex{X,\eta} := \cindex{S,\eta} \).  

\end{definition}

The Conley index enjoys several properties such as invariance under continuous deformation (also known as the continuation property). We will later need the following definition of a product flow.

\begin{definition} Let \(\eta_1 , \eta_2\) be a flow on \(\Omega_1 , \Omega_2\) respectively. A product flow \(\eta_1 \times \eta_2 \) is  a flow on \(\Omega_1 \times \Omega_2 \) defined by $(x_1 , x_2 ) \cdot t \mapsto (\eta_1(x_1,t),\eta_2 (x_2,t))$.
\end{definition}

It is not hard to check that, for $i=1,2$, if $X_i$ is an isolating neighborhood of $S_i$ with respect to the flow $\eta_i$ on $\Omega_i$, then $X_1 \times X_2$ is an isolating neighborhood for $S_1 \times S_2$ with respect to the product flow. Moreover, if $(N_i ,L_i )$ is an index pair for $S_i$, then $(N_1 \times N_2 , N_1 \times L_2 \,\cup \, L_1 \times N_2)$ is an index pair of $S_1 \times S_2$. Consequently, the Conley index of the product $S_1 \times S_2$ is given by a smash product $(N_1/L_1 , [L_1]) \wedge (N_1/L_1 , [L_1])$.

We give some simple examples of index pairs and Conley indices. 

\begin{example} 
Let $V$ be a finite-dimensional vector space and $L$ be a self-adjoint linear map on \(V\) to itself.  Consider a flow $\eta$ given by a formula $\eta(v,t) = e^{-tL} v$. A flow of this form is sometimes known as a \emph{linear flow} because $\eta (v,t)$ is an integral curve of  the ODE  
\begin{align} \label{eq linearflow}
\frac{\partial}{\partial t} \eta (v,t) &=  - L (\eta(v,t)) .
\end{align}
For simplicity, we will assume that $L$ has no kernel.

One can see that $\left\{0\right\}$ is an isolated invariant set of this flow. Let us decompose $V = V^+ \oplus V^-$ to the positive eigenspace and the negative eigenspace with respect to $L$. One can check that $( B(V^+) \times B(V^-) , B(V^+ ) \times S(V^- ) )$ is an index pair for $\left\{ 0 \right\}$, where we denote by $B(W )$ and $S(W)$ a unit disk and a unit sphere in $W$ respectively. Consequently, the Conley index has a homotopy type of $B(V^- ) / S(V^- )$, which can be identified with  $S^{V^-}$ the compactification of $V^-$ with a basepoint at infinity.

\label{ex linearflow} 
\end{example} 

\begin{remark} The negative sign in (\ref{eq linearflow}) is introduced to follow the downward-gradient flow convention. In particular, the flow $\eta$ is the downward gradient flow of a functional \(f(v) = \frac{1}{2}\langle Lv , v \rangle \) on \(V\).
\end{remark}


\section{Flows generated by permissible vector fields on a Hilbert space} \label{sec vecf}
 
\subsection{Permissible vector fields} For the rest of the paper, we let $H$ be a Hilbert space. A vector field on $H$ is a continuous map from $H$ to itself. We will be interested in a special class of vector fields. 

\begin{definition} \label{defgoodvf}
We say that a vector field $F : H \rightarrow H$ is \emph{\goodvf}if $F$ admits a decomposition $F = L + Q$ such that
\begin{enumerate}
        \item $L$ is a bounded self-adjoint Fredholm operator on $H$.   
        \item $Q$ is locally Lipschitz and compact (possibly nonlinear).
        \item there exist positive constants $c_{1},c_{2}$ so that $\left\| Q(x) \right\| \leq c_{1} \left\| x \right\| + \frac{c_{2}}{1 + \left\|x\right\|}$ for all $x \in H$
         
\end{enumerate}
We will also say that a pair $(L,Q)$, or simply $F = L+Q$, is a \emph{\goodvf decomposition} for $F$ if it satisfies the above conditions. In addition, $L$ and $Q$ will be referred as a linear part and a compact part of $F$ respectively.

\end{definition}

Given two different \goodvf decomposition $F = L_1 + Q_1 = L_2 + Q_2$, we see that the difference $L_1 - L_2$ is compact. In fact, this gives a one-to-one correspondence between a set of \goodvf decompositions of $F$ and a space of linear self-adjoint compact operators. 

We will also extend this notion to a family of vector fields. 

\begin{definition}
Let $\Lambda$ be a metric space, regarded as a parameter space. We say that a family of vector fields $F: H \times \Lambda \rightarrow H$ is a \emph{\goodvf family} of vector fields if there is a decomposition $F(x,\lambda) = L(x, \lambda ) + Q(x,\lambda)$, where we sometimes write $F_\lambda$ for the restriction $F(\cdot , \lambda )$ to a point in $\lambda \in \Lambda$,
satisfying the following properties
\begin{enumerate}
        \item For each $\lambda$, the decomposition $F_\lambda  = L_\lambda  + Q_\lambda$ is a \goodvf decomposition.
        \item $L_\lambda$ is a continuous family in the norm topology of bounded linear operators.
        \item $Q : H \times \Lambda \rightarrow H$ is compact.
        \item There exist positive constants $C_{1},C_{2}$ so that $\left\| Q(x, \lambda) \right\| \leq C_{1} \left\| x \right\| + \frac{C_{2}}{1 + \left\|x\right\|}$   for all $(x,\lambda) \in H \times \Lambda $. \end{enumerate}

\end{definition}

\begin{remark} The class of permissible vector fields we consider and the class of \(\mathcal{LS}\)-vector fields studied in \cite{Geba} are slightly different. We point out some distinction.\begin{enumerate}
        \item The linear part $L$ of a \(\mathcal{LS}\)-vector field need not be self-adjoint. Our Fredholm requirement resembles their condition that the spectrum of $L$ is isolated from the imaginary axis in the complex plane. 
        \item The condition for a \(\mathcal{LS}\)-vector field that $H$ is a sum of finite-dimensional eigenspaces of $L$ can be compared to a condition that \(L\) admits an \exhaust sequence (Definition~\ref{def exhaust}). However, we will  only need this condition later in the paper (see Lemma~\ref{lemma cofinal}).
        \item The last condition in Definition~\ref{defgoodvf} could be omitted. When studying a flow on a fixed bounded subset, one can always use a cut-off function on this bounded subset to make $F$ satisfy this condition without changing a flow inside this subset. 
\item For a family of vector fields, our definition allows the linear part $L_\lambda$ to vary along the parameter space. 
\end{enumerate}
\end{remark}

We will now consider a family of flows on $H$ generated by a family of vector fields $F$. This is a family of flows $\eta : H \times \mathbb{R} \times \Lambda \rightarrow H$ whose trajectory is an integral curve of \(F\), i.e. the solution of the following ODE
\begin{align*} 
\frac{\partial}{\partial t} \eta(x,t,\lambda) &= -F( \eta(x,t,\lambda), \lambda) \\
\eta(x,0,\lambda) &= x
        \end{align*}
        
\begin{note} The last condition in Definition~\ref{defgoodvf} implies that $F$ is \emph{subquadratic} i.e. there exist positive constants $c_1, c_2$ such that $2 \left|\left\langle F(x) , x \right\rangle \right| \leq c_{1} \left|x \right|^2 + c_{2} $   for all $x \in H$. The subquadratic condition guarantees that $\eta$ is defined for all time $t$ (cf.~\cite{TW}). Otherwise, the flow is only a local flow. \end{note}


The argument from the proof of Proposition~2.3 from \cite{Geba} can be extended to  show the following properness result for maximal invariant sets of flows generated by \goodvf vector fields.

\begin{proposition} \label{properh} Let $X$ be a closed and bounded subset of $H$ and $\eta$ be a family of flows generated by a \goodvf family of vector fields $F = L + Q : H \times \Lambda \rightarrow H$. Then, the projection to second factor $pr_2 : \operatorname{Inv}(X \times \Lambda , \eta) \subset H \times \Lambda \rightarrow \Lambda$ is proper.
\end{proposition}

\begin{proof} Let $\left\{(x_n,\lambda_n)\right\}$ be a sequence in $\operatorname{Inv}(X \times \Lambda)$ with $\lambda_n \rightarrow \lambda$. Let $H = H_+ \oplus H_- \oplus H_0$ be the spectral decomposition corresponding to positive, negative, kernel part of $L_\lambda$ respectively. Let $\pi_\pm , \pi_0$ be the orthogonal projection from $H$ onto $H_\pm , H_0$. 

We will show that the sequence $\left\{ x_n \right\}$ has a Cauchy subsequence by decomposing $x_n$ with respect to $H_\pm$ and  $H_0$. Since the set $\operatorname{Inv}(X \times \Lambda)$ is a closed subset of a complete space, the Cauchy subsequence will be also convergent.    

Since $L_\lambda$ is a self-adjoint Fredholm operator, there is $\delta > 0$ such that the interval $(- \delta , \delta )$ contains no spectrum of $L_\lambda$ except possibly $0$.  Then we have that
\[
        \| e^{t L_\lambda} x \| \geq e^{t \delta}  \|x\|  \ \ \mbox{for all $x \in H_+$}. 
\]

Now let $\epsilon >0$ be arbitrary. Since $X$ is bounded, we assume that $X \subset B(R)$ a ball of radius $R$. We choose $T > 0$ so that $e^{T \delta} > \frac{3R}{\epsilon}$. Using an integral equation, we can write a formula for $\eta$ as 
\begin{align}  \eta(x,t,\lambda) = e^{t L_\lambda} x + e^{t L_\lambda} \int^{t}_{0} e^{-\tau L_\lambda} Q(\eta(x,\tau,\lambda),\lambda) d\tau  .  \label{eq intequa}
\end{align}

We set $U(x,t,\lambda) := e^{t L_\lambda} \int^{t}_{0} e^{-\tau L_\lambda} Q(\eta(x,\tau,\lambda),\lambda) d\tau $. 
\begin{lemma} \label{lemma Ux_n} The sequence $U(x_n , T , \lambda_n )$ has a Cauchy subsequence (cf. \cite[Proposition~A.18]{Rab}). \end{lemma}

\begin{proof} Since $F$ is subquadratic, we have
\[ \frac{\partial}{\partial t} \left\| \eta(x,t,\lambda) \right\|^2 = 2 \left\langle F( \eta(x,t,\lambda), \lambda) , \eta(x,t,\lambda) \right\rangle \leq C_{1} \left\| \eta(x,t,\lambda) \right\|^2 + C_{2}
\]
and consequently
\[ \left\|\eta(x,t,\lambda) \right\|^2 \leq e^{C_{1} t} \left\|x  \right\|^2  + \frac{C_{2}}{C_{1}} (e^{C_{2}t} - 1)
\]
for positive time $t$. Then the set $\left\{  ( \eta(x_n,\tau,\lambda_n) , \lambda_n ) : n \in \mathbb{N} \text{ and } 0 \leq \tau \leq T \right\}$ is a bounded subset of $H \times \Lambda$, and so its image under $Q$ is precompact. Consequently the set \[\mathcal{K}\ = \left\{  T e^{-\tau L_{\lambda_n}} Q(\eta(x_n,\tau,\lambda_n),\lambda) : n \in \mathbb{N} \text{ and } 0 \leq \tau \leq T \right\}\] is also precompact.

We recall the fact that, in a Hilbert space, a convex hull of a precompact set is precompact. Then the integral \(\int^{t}_{0} e^{-\tau L_\lambda} Q(\eta(x,\tau,\lambda),\lambda) d\tau\) lies in the convex hull of \(\mathcal{K}\), which is precompact. Consequently, the sequence $U(x_n , T , \lambda_n )$ has a Cauchy subsequence as claimed. 

\end{proof}

By the above Lemma, we can pass to a subsequence such that $\left\{U(x_n,T,\lambda_n)\right\}$ is Cauchy. Since $\left\{(x_n,\lambda_n)\right\}$ is an element of the maximal invariant set, a point $\eta(x_n,T,\lambda_n)$ must lie in $X \subset B(R)$ as well. Thus, 
\begin{eqnarray*} \| e^{T L_\lambda} (x_m - x_n) \| & \leq & \| e^{T (L_\lambda - L_{\lambda_m})} x_m \| + \| e^{{T (L_\lambda - L_{\lambda_n})}} x_n \| + \| \eta(x_m,\lambda_m,T) \| \\  && + \| \eta(x_n,T,\lambda_n) \| + ||U(x_m,T,\lambda_m)- U(x_n,T,\lambda_n) \| \\
&\leq & 3R \ \ \mbox{for $m,n$ sufficiently large.}
\end{eqnarray*}

On the other hand, we have
\[ || e^{T L_\lambda} (x_m - x_n) || \geq || e^{T L_\lambda} \pi_+ (x_m - x_n) || \geq \frac{3R}{\epsilon} || \pi_+ (x_m) - \pi_+ (x_n) || .
\]
Combining this with the previous inequality, we see that the sequence $\{ \pi_+ (x_n) \}$ has a Cauchy subsequence.

Using the same argument, one can show  that the sequence $\{ \pi_- (x_n) \}$ also has a Cauchy subsequence. The sequence $\{ \pi_0(x_n) \}$ has a Cauchy subsequence because it is a bounded sequence in a finite-dimensional Euclidean space. 

Therefore the sequence $\left\{(x_n,\lambda_n)\right\}$ has a Cauchy subsequence and we finish the proof. 

\end{proof} 

A special case for a permissible family of vector fields is a family obtained from a sequence of vector fields with an appropriate limit. We can identify and topologize $\mathbb{N}_\infty = \mathbb{N} \cup \left\{ \infty\right\}$ using a subspace $\left\{ \frac{1}{n} | n \in \mathbb{N} \right\} \cup \left\{ 0 \right\}$ of the interval $\left[ 0 , 1\right]$ with the standard topology. 
It is straightforward to check the following compactness property for a family of maps parametrized by $\mathbb{N}_\infty$.

\begin{lemma}  A map $Q : H \times \mathbb{N}_\infty \rightarrow H$ is compact if $Q(\cdot , n) : H \rightarrow H$ is compact for each $n$ and $Q(\cdot , n)$ converges to $Q(\cdot , \infty)$ pointwise uniformly on any fixed bounded subset of $ H$.
\label{lem Qconverge}          
\end{lemma}

\subsection{Isolating neighborhoods and Continuity} \label{sec continuity}

Here, we will reintroduce the notion of an isolating neighborhood. 
\begin{definition} For a flow $\eta$ on a Hilbert space \(H\), a closed and bounded subset $X$ of $H$ is called \emph{an isolating neighborhood} if $\Inv{X,\eta} \subset \Inter{X}$. 
\end{definition}

As a main consequence of Proposition~\ref{properh}, we can deduce an infinite-dimensional version of continuity of isolating neighborhoods for a family of flows generated by a permissible vector fields. Denote the flow $\eta(\cdot,\cdot,\lambda) : H \times \mathbb{R} \rightarrow H$ by $\eta_\lambda$. 

\begin{corollary} \label{lambdaopen} Let $\eta$ be a family of flows generated by a \goodvf family of vector fields. Then, for a fixed closed and bounded subset \(X\) of \(H\), the set \{$\lambda \in \Lambda : X$ is an isolating neighborhood for the flow $\eta_\lambda$\} is open.  
\end{corollary}

\begin{proof} We will show that the compliment of this set is closed. Let $\left\{\lambda_n\right\} $ be a sequence in $\Lambda$ such that there exists $\left\{x_n\right\}$ with $x_n \in \operatorname{Inv}(X, \eta_{\lambda_n}) \cap \partial X$  and $\lambda_n \rightarrow \lambda$. By properness of the projection $pr_2$ from Proposition~\ref{properh}, there is a subsequence of $\left\{x_n\right\}$ with $x_n \rightarrow x$ for some $x \in H$. Since $\eta$ is continuous and $X$ is closed, we see that $x \in \operatorname{Inv}(X, \eta_{\lambda}) \cap \partial X$.   
\end{proof}



We will now fix a flow $\eta$ on $H$ generated by a \goodvf vector field $F$ with a \goodvf decomposition $F = L + Q$  as well as an isolating neighborhood $X$ with respect to \(\eta\). 

From Corollary~\ref{lambdaopen}, we know that $X$ is also an isolating neighborhood for flows in a neighborhood of $\eta$ whenever $\eta$ is a part of a family of permissible vector fields. To refine this result, we will introduce a (pseudo)metric between permissible decompositions so that $X$ is an isolating neighborhood for flows generated by nearby vector fields.

For compact maps $Q_1$ and $Q_2$, since $X$ is bounded, we define a pseudometric which depends on the set $X$ by
\begin{align}   \rho_{X} (Q_1 , Q_2) = \sup_{x \in X} \left\| (Q_1 - Q_2)x  \right\| . \nonumber \end{align}
We now define a pseudometric between two \goodvf decompositions $F_1 = L_1 + Q_1$ and $F_2 = L_2 + Q_2$ by
\begin{align} \label{pseudo1}   \bar{\rho}_{X} (L_1 + Q_1, L_2 + Q_2) = \left\| L_1 - L_2 \right\| + \rho_{X} (Q_1 , Q_2). \end{align}
Note that this does not quite measure distance between vector fields. For example, even when $F= L_1 + Q_1 = L_2 + Q_2 $, the quantity $\bar{\rho}_{X} (L_1 + Q_1, L_2 + Q_2)$ is nonzero in general. To modify this definition, one could take the infimum over all permissible decompositions of a vector field.  

As another consequence of Proposition~\ref{properh}, we can show that 

\begin{proposition} \label{neighborisolate} There is $\epsilon_{0} > 0$ such that $X$ is an isolating neighborhood for any flow $\eta'$ generated by $F' = L' + Q'$ with $\bar{\rho}_{X} (L' + Q', L + Q) < \epsilon_{0}$.
\label{prop hilbertpseudo}
\end{proposition}

\begin{proof} Suppose the statement is false. There would be a sequence of \goodvf decomposition $F_n = L_n + Q_n$ such that $\bar{\rho}_{X} (L_n + Q_n , L + Q) < \epsilon_n$ and $\epsilon_n \rightarrow 0$. There would also exist a sequence $\left\{ x_n \right\}$ such that $x_n \in \Inv{ X , \eta_n }$ and $x_n$ lies in the boundary $\partial X$ , where $\eta_n$ is generated by $F_n$. We will show that the sequence $\left\{ x_n \right\}$ is Cauchy and arrive at a contradiction because its limit lies in the intersection $ \operatorname{Inv}(X, \eta) \cap \partial X$.

Set
\begin{align}
  U_n (x,t) &= e^{t L} \int^{t}_{0} e^{-\tau L_n} Q_n (\eta_n (x,\tau)) d\tau , \nonumber \\
        V_{n} (x,t) &= e^{t L} \int^{t}_{0} e^{-\tau L} Q(\eta_n (x,\tau)) d\tau . \nonumber
\end{align}
By compactness of $Q$ and boundedness of \(X\), we can show  that a sequence $\left\{ V_{n} (x_n ,t) \right\}_n$ is Cauchy  for fixed $t$ analogous to the proof of Lemma~\ref{lemma Ux_n}. Since $\bar{\rho}_{X} (L_n + Q_n , L + Q) $ goes to $0$, we see that $\left\| U_n (x_n,t) - V_{n} (x_n,t) \right\| $ also goes to $0$, so a sequence $\left\{ U_n (x_n ,t)\right\}$ is also Cauchy. 

Using an integral equation for flows similar to (\ref{eq intequa}) and invariance of $x_n$, one can see that the quantity $\left\|  e^{t L} (x_m - x_n) \right\|$ is uniformly bounded. With an argument similar to the proof of Proposition~\ref{properh},  we can deduce that $\left\{ x_n \right\}$ is Cauchy and finish the proof. 

\end{proof}


\section{The stable Conley index}

\subsection{Compression of vector fields} \label{sec compress}

In this section, let us fix a flow $\eta$ on $H$ generated by a \goodvf vector field $F$.
We want to consider its approximated flow on a finite-dimensional subspace $V$ of $H$ in order to apply Conley index theory.

For now, let us also fix $F = L + Q$ a \goodvf decomposition. Denote $\pi_V$ by the orthogonal projection from $H$ onto $V$. We will consider in a vector field of the form
\begin{align} F_V = \pi_V L \pi_V + (1-\pi_V) L (1-\pi_V ) + \pi_V Q , \label{eq F_V} \end{align}
as well as a flow $\eta_V$ generated by $F_V$. This vector field $F_V$ can be considered as a \emph{compression} of $F$ on $V$. We will regard the term $\pi_V L \pi_V + (1-\pi_V) L (1-\pi_V )$ as a linear part of a decomposition of $F_V$. 

Viewing $L$ as a block matrix on $H = V \oplus V^\bot$, the linear part of $F_V$ consists of the diagonal blocks. Whereas, the difference
\begin{align}  \label{diffLV} L -  \pi_V L \pi_V - (1-\pi_V) L (1-\pi_V )  = \pi_V L (1- \pi_V ) + (1-\pi_V) L \pi_V  \end{align}
consists of the antidiagonal blocks. Note that this difference  is finite rank (and hence compact), so the linear part of $F_V$ is also Fredholm. It is then straightforward to check that $F_V$ is also a \goodvf vector field.

Our goal is to find a condition for \(V\) so that $F_V$ is sufficiently close to \(F\) either as a \goodvf family of vector fields or in the sense of Proposition \ref{neighborisolate}. In the latter sense, we use the pseudometric defined in (\ref{pseudo1}) to write down relevant  quantities for \(V\).

For the linear parts, we have an operator norm of the difference from (\ref{diffLV}) as
\begin{align} \left\| \pi_V L (1- \pi_V ) + (1-\pi_V) L \pi_V \right\| &= \left\| \pi_V L (1- \pi_V ) + (1-\pi_V) L \pi_V \right\| \notag \\
&= \left\| \pi_V L (1- \pi_V ) - (1-\pi_V) L \pi_V \right\| \notag \\
&= \left\|\pi_V L - L \pi_V \right\|, \label{eq commute}
\end{align}
where we can switch the sign in the second line because  $\pi_V L (1- \pi_V )$ and $(1-\pi_V) L \pi_V$ have orthogonal domains and codomains. As a result, this difference has the same norm as the commutator \([L,\pi_V]\) and vanishes when \(V\) is an \(L\)-invariant subspace.

For the compact parts, we have to look at the difference $ (1- \pi_{V})Q  $ on a bounded subset of \(H\). Recall the following fact:

\begin{lemma} \label{lemma cptpointwise}Suppose that $\left\{ V_n \right\}$ is a sequence of subspaces of $H$ such that $\pi_{V_n} \rightarrow 1 $ pointwise and $Q : H \rightarrow H$ is a compact map. Then $\pi_{V_n} Q$ converges to $Q$ pointwise uniformly on any bounded set. 
\end{lemma}

\begin{proof} We will prove this by contradiction. Let $B$ be a bounded subset of $H$ and suppose there exists a sequence $\left\{ x_j \right\}$ in $B$ such that $\left\| (1- \pi_{V_{n_j}} ) Q (x_j )  \right\| > \delta $ and $n_j$ goes to infinity.  

Since $\left\{ x_j \right\}$ is bounded, $Q(x_j ) $ converges to some $y$ in $H$. Then
\begin{eqnarray*} \left\| (1- \pi_{V_{n_j}} ) Q (x_j )  \right\|  & \leq &  \left\| (1- \pi_{V_{n_j}} ) (Q (x_j ) - y )  \right\|  + \left\| (1- \pi_{V_{n_j}} ) y  \right\|  \\
&\leq& \left\| Q (x_j ) - y   \right\| + \left\| (1- \pi_{V_{n_j}} ) y  \right\|  \\
&<& \delta \ \ \mbox{for sufficiently large $j$}
\end{eqnarray*}  
and we reach a contradiction.

\end{proof}
Thus, on a fixed bounded subset \(X\), we have $\rho_{X} (Q , \pi_V Q) \rightarrow 0$ when $\pi_{V_n} \rightarrow 1 $ pointwise. 
From these observations regarding linear and compact parts, we introduce a sequence of subspaces that are suitable for finite-dimensional approximation.

\begin{definition} Let $L_0$ be a bounded linear operator and $\left\{ V_n \right\}$ be a sequence of finite-dimensional subspaces of $H$. We say that the sequence $\left\{ V_n \right\}$ is an \emph{\exhaust}sequence with respect to $L_0$ if it is an increasing sequence satisfying  
\begin{enumerate}
  \item $[ L_0 ,\pi_{V_n} ] \rightarrow 0$ in operator norm,
        \item $\pi_{V_n} \rightarrow 1$ pointwise (i.e. in strong operator topology).   
\end{enumerate} 
\label{def exhaust}
\end{definition}

The asymptotically-invariant part is referred to the first condition as a subspace $V_n$ is behaving more and more like an invariant subspace. 

One can check that if a sequence $\left\{ V_n \right\}$ is \exhaust with respect to $L$, then it is also an \exhaust with respect to \(L + K\) when $K$ is linear and compact. Hence, if we fix a \goodvf vector field, the definition of an \exhaust sequence is actually independent of the choice of \goodvf decompositions.

\begin{remark} An example of an \exhaust sequence is a sequence of increasing eigenspaces. This is the case when $L$ arises from a self-adjoint elliptic differential operator. However, it is not clear if an \exhaust sequence always exists for a general linear Fredholm operator $L$. The property  is invariant under compact perturbation.  
\end{remark} 

One of the main motivations for the definition of an \exhaust sequence is to create a permissible family of vector fields. From Lemma~\ref{lem Qconverge} and Lemma~\ref{lemma cptpointwise}, we see that: 
\begin{lemma} Let $\left\{ V_n \right\}$ be an \exhaust sequence of subspaces with respect to $L$ and let $F_n$ be a family of vector fields parametrized by $\mathbb{N}_\infty$ where 
\[ F_n = F_{V_n} = \pi_{V_n} L \pi_{V_n} + (1-\pi_{V_n}) L (1-\pi_{V_n} ) + \pi_{V_n} Q , \]
and $F_\infty = F$. Then, the family $F_{\mathbb{N}_\infty}$ is a \goodvf family of vector fields.
\end{lemma}

%
%
%
%
%
Now let $X$ be a fixed isolating neighborhood with respect to the flow $\eta$. The above Lemma and the continuity result for isolating neighborhoods (Corollary~\ref{lambdaopen}) implies that \(X\) is an isolating neighborhood for \(\eta_{V_n}\) when $\left\{ V_n \right\}$ is an \exhaust sequence of subspaces with respect to $L$ and \(n\) is sufficiently large. These subspaces \(V_n\) provide ground for finite dimensional approximation. 

Alternatively, we can switch our viewpoint to the context of Proposition~\ref{neighborisolate}. As we fixed \(F=L+Q\) and \(X\), there is \(\epsilon_0\) such that $X$ is an isolating neighborhood for a flow generated by a vector field $F' = L' + Q'$ such that  $\bar{\rho}_{X} (L' + Q', L + Q) < \epsilon_{0}$. From (\ref{eq commute}), when we plug in $F_V$ for $F'$, the previous inequality becomes 
\begin{align}\label{eq pseudov1} \|\pi_V L - L \pi_V \| +\sup_{x \in X} \left\| (1 - \pi_V)Qx  \right\| < \epsilon_0 . \end{align}
We will see that a subspace $V$ satisfying the above inequality is also suitable for finite dimensional approximation.


When $X$ is an isolating neighborhood for the flow $\eta_V$, we can see that a compact subset $X \cap V \subset V$ is an isolating neighborhood with respect to the flow $\eta_V$ restricted to $V$. Since $V$ is finite-dimensional, we can now apply Conley theory to obtain the Conley index of \(X \cap V\) denoted by
\[ \cindex{X \cap V , \eta_V }. \] 
This gives us a collection of Conley indices from each finite-dimensional subspace \(V\) whenever \(X\) is an isolating neighborhood for \(F_V\). Note that the vector field \(F_V\) restricted to \(V\) is equal to $\pi_V F$. Consequently, the flow $\eta_V$ restricted to $V$ does not depend on a choice of decompositions of $F$. 

We will now establish a relationship between Conley indices coming from compressed flows on different subspaces. 
Let $V, W$ be finite-dimensional subspaces of $H$ with orthogonal decomposition $W = V \oplus U$. Roughly speaking, when \(W\) and $V$ are sufficiently large, a flow \(\eta_W\) on \(W\) can be related to a product of a flow \(\eta_V\) on \(V\) and a linear flow generated by \(L\) on \(U\).

To start with, we consider a flow $\eta_{V,W}$ generated by a vector field $F_{V,W}$ of the form
\begin{align}
F_{V,W} =  \pi_V L \pi_V + \pi_U L \pi_U + (1-\pi_W) L (1-\pi_W ) + \pi_V Q , \label{eq F_V,W} \end{align}
where we will regard \(\pi_V L \pi_V + \pi_U L \pi_U + (1-\pi_W) L (1-\pi_W )\) as a linear part of \(F_{V,W}\).

\begin{lemma} \label{lemma F_V,W} Let \(W = V \oplus U\) be a subspace of \(H\). There is an inequality \[ \|L -  \pi_V L \pi_V - \pi_U L \pi_U -(1-\pi_W) L (1-\pi_W )\| \leq \left\| L \pi_W - \pi_W L \right\| + \left\| L \pi_V - \pi_V L \right\|.  \]
\end{lemma}
\begin{proof}
Since $\pi_W = \pi_V + \pi_U $, we have an identity \[ \pi_W L \pi_W =  \pi_V L \pi_V + \pi_U L \pi_U + \pi_U L \pi_V + \pi_V L \pi_U. \]
Consider the term \(\pi_U L \pi_V + \pi_V L \pi_U\), we see that \(\pi_U L \pi_V \) and \( \pi_V L \pi_U\) are adjoint of each other. In particular, their norms are equal and \(\left\| \pi_U L \pi_V + \pi_V L \pi_U \right\| = \sqrt{2} \left\| \pi_U L \pi_V \right\| \) because \(U\) and \(V\) are orthogonal. We can also apply this when \(U = V^\bot \) so that
\begin{align} \sqrt{2} \left\| (1-\pi_V) L \pi_V \right\| &= \left\| (1-\pi_V) L \pi_V  + \pi_V L (1-\pi_V ) \right\| \notag \\
&= \left\| (1-\pi_V) L \pi_V  - \pi_V L (1-\pi_V ) \right\| \notag \\   
&= \left\| L \pi_V - \pi_V L \right\| . \label{eq sqrt2LV}
\end{align}

Then we have  
\begin{align*}
        & \left\| L - \pi_V L \pi_V - \pi_U L \pi_U - (1-\pi_W) L (1-\pi_W ) \right\| \\
        =& \left\| L \pi_W + \pi_W L - \pi_W L \pi_W   - \pi_V L \pi_V - \pi_U L \pi_U \right\|  \\
        \leq& \left\| L \pi_W + \pi_W L - 2\pi_W L \pi_W \|+ \|  \pi_W L \pi_W   - \pi_V L \pi_V - \pi_U L \pi_U \right\|  \\
        = & \left\| (1-\pi_W) L \pi_W  + \pi_W L (1-\pi_W ) \right\| + \left\| \pi_U L \pi_V + \pi_V L \pi_U \right\| \\
        = & \left\| L \pi_W - \pi_W L \right\| + \sqrt{2} \left\| \pi_U L \pi_V \right\| \\
        \leq & \left\| L \pi_W - \pi_W L \right\| + \sqrt{2} \left\| (1-\pi_V) L \pi_V \right\| \quad(\text{since } U \subset V^{\bot})   \\
        = & \left\| L \pi_W - \pi_W L \right\| + \left\| L \pi_V - \pi_V L \right\|.
\end{align*}
\end{proof}
Consequently, we see that the vector field \(F_{V,W}\) is sufficiently close to \(F\) in the pseudometric \(\bar{\rho}_{X}\) whenever \(F_V\) and \(F_W\) are. In particular, a ball with respect to the pseudometric \(\bar{\rho}_{X}\) is convex so that one can make a convex combination \((1-s) F_{V,W} + s F_W \) sufficiently close to \(F\). 

The difference of \(L\) and a convex combination of  linear parts of \(F_{V,W}\) and \(F_W\) is always finite rank, and hence compact. 
Note that, in general, a segment joining two Fredholm maps might not lie in the space of Fredholm maps. However, we  can still choose the two linear parts to lie in a small ball of $L$ contained in the space of Fredholm maps.

From the above discussion, we choose \(V \) and $W$ so that \((1-s) F_{V,W} + s F_W \) is a \goodvf family of vector fields parametrized by $s \in I = [0,1]$. From the previous lemma, if we require that $\left\| L \pi_W - \pi_W L \right\| + \left\| L \pi_V - \pi_V L \right\| < \epsilon_0$, we can ensure that $X$ is an isolating neighborhood for the family  \((1-s) F_{V,W} + s F_W \). With an additional hypothesis, we have the following result, which is a slight generalization of arguments in \cite{Geba} and \cite{Man1}.


\begin{proposition} Let $W = V \oplus U$ be a finite-dimensional subspace. 
 Suppose that $X \times I $ is an isolating neighborhood for the family of flows generated by \((1-s) F_{V,W} + s F_W \) parametrized by $s \in I$ and the bilinear form \(\pi_U L_{|U}\) is nondegenerate. Then
\begin{align} \cindex{X \cap W , \eta_W} \cong \cindex{ X \cap V , \eta_V} \wedge S^{U^-},  \label{eq indexsusp} \end{align} 
where \(U^-\) is the negative eigenspace of \(\pi_U L_{|U}\). In fact, we can replace \(U^-\) by any maximal negative definite subspace with respect to \(\pi_U L_{|U}\).  
\label{prop suspindex}
\end{proposition}
\begin{proof} 
We see that the vector field \((1-s) F_{V,W} + s F_W \) is $W$-invariant for each $s$, so we can restrict to a family of flows on $W \times I$ with $(X \cap W) \times I$ becoming a compact isolating neighborhood for this family. By continuation of the Conley index, we have 
\[ \cindex{X \cap W , \eta_W} \cong \cindex{X \cap W , \eta_{V,W}}. \] 

Next, we look at the flow $\eta_{V,W}$ on $W$. The vector field $F_{V,W}$ restricted to $W$ becomes $\pi_V L \pi_V + \pi_U L \pi_U + \pi_V Q$, which is
\[ \begin{bmatrix}\pi_V F & \pi_V Q \\
0 & \pi_UL \\
\end{bmatrix} \]
as a block matrix on $W = V \oplus U$. We see that the $U$-component of the flow is given by a nondegenerate linear flow. In particular, an element with nonzero \(U\)-component cannot stay in a bounded set for all time. Consequently, an element of the maximal invariant subset of \(X \cap W\) with respect to \(\eta_{V,W}\) must lie in \(V \), in particular $\Inv{X \cap W , \eta_{V,W}} = \Inv{X \cap V  , \eta_{V,W}} $. 

Denote by $B(U , R)$  the ball of radius $R$ in $U$ centered at the origin. By similar argument, we can check that the set $N = (X \cap V ) \times B(R,U)$  is an isolating neighborhood for $\eta_{V,W}$ with the same maximal invariant subset $\Inv{N, \eta_{V,W}} = \Inv{X \cap V  , \eta_{V,W}}=\Inv{X \cap W , \eta_{V,W}} $. Thus, their Conley indices are the same, i.e. 
\[ \cindex{X \cap W , \eta_{V,W}}  = \cindex{N, \eta_{V,W}}. \]


The final step is to compare the Conley index of \(N\) to that of a product flow. Consider a family of flows $\widehat{\eta}_s$ on $W$ parametrized by \(s \in\mathbb{R}\) which is generated by a vector field of the form  
\[ \widehat{F}_s (w) =  \pi_V L \pi_V (x) + \pi_U L (y) + \pi_V Q (x + sy) , \] 
where we use a decomposition $w = x+y$ with $x \in V $ and $y \in U$. We see that \(\widehat{F}_1 \) is equal to \(F_{V,W} \) on \(W\) and \(\widehat{F}_0 \) is given by
\[ \begin{bmatrix}\pi_V F  & 0 \\
0  & \pi_{U} L  \\
\end{bmatrix} \]
 as a block matrix on $W = V \oplus  U$. Similarly, we can conclude that the maximal invariant subset $\Inv{N,\widehat{\eta}_s}$ lies in $V$ for all $s$. We notice that the vector field $\widehat{F}_s$ on $V$ is equal to $\pi_V F$, so the maximal invariant set of \(N\) is the same for each \(s \in \mathbb{R} \). In other words, $N$ is an isolating neighborhood for the each of the flow $\widehat{\eta}_s$ with the same maximal invariant subset. By continuation of the Conley index, we have that 
\[  \cindex{N , \widehat{\eta}_1} \cong \cindex{N , \widehat{\eta}_0}.
\]


From the above block matrix, we see that the flow \(\widehat{\eta}_0\) is the product of the flow \(\eta_V\) on \(V\) and a linear flow on \(U\). Hence,
\[ \cindex{N , \widehat{\eta}_0} \cong \cindex{ X \cap V  , \eta_V} \wedge \cindex{B(R,U) , e^{\pi_{U} L }}.
\] 
From Example~\ref{ex linearflow}, we have that that the index \(\cindex{B(R,U) , e^{\pi_{U} L }}\) has a homotopy type of \(S^{U^-}\) where \(U^-\) is the negative eigenspace of \(\pi_U L_{|U}\).  

Putting everything together, we have the desired result.

\end{proof}

\subsection{Construction of the stable Conley Index} \label{sec stable}
In this context, one natural choice of stable homotopy categories for developing the stable Conley index is a notion of coordinate-free spectra. A background in stable homotopy theory can be found in \cite{May2}. Naively, a (pre)spectrum is a sequence of spaces related by some maps whereas a coordinate-free spectrum is a collection of spaces indexed by finite-dimensional subspaces of a fixed infinite-dimensional vector space, called a \emph{universe}. In this section, we use our fixed Hilbert space $H$ for a universe.

\begin{definition} A coordinate-free prespectrum \(E\) is a collection of pointed spaces \(E_V\) for each finite-dimensional subspace \(V\) of \(H\) together with structure maps
\[ \Sigma^U E_V \rightarrow E_W \]
whenever \(W = V \oplus U \).
\end{definition}  

\begin{note} A prespectrum is a spectrum when the adjoint map \(E_V \rightarrow \Omega^U E_W \) is a homeomorphism. For simplicity, we will assume we are working on spectra as one can always apply the spectrification functor to turn a prespectrum into a spectrum.
\end{note}

We will first fix an isolating neighborhood \(X\) of a flow $\eta$ on $H$ generated by a \goodvf vector field $F = L + Q$ as in the previous subsection. 
Let us now consider a relation between Conley indices from (\ref{eq indexsusp})
 \begin{align} \cindex{X \cap W , \eta_W} \cong \cindex{ X \cap V , \eta_V} \wedge S^{U^-} .  \nonumber
 \end{align}
Observe that this relation has the term \(S^{U^-}\) rather than \(S^U\) in the above definition of spectra.

To obtain a prespectrum from Conley indices, we could try to assign \(\Sigma^{V^+} \cindex{ X \cap V , \eta_V} \) to a subspace \(V\) so that we have
\[ \Sigma^{U}\left(\Sigma^{V^+} \cindex{ X \cap V , \eta_V}  \right) \cong \Sigma^{W^+}\left( \Sigma^{U^-}\cindex{ X \cap V , \eta_V}\right) \cong \Sigma^{W^+}\cindex{X \cap W , \eta_W}. \]
However, the Conley index \(\cindex{ X \cap V , \eta_V}\) is not defined for every finite-dimensional subspace \(V\) of \(H\).  Another issue is that a Conley index is not a space, but actually a collection of spaces with specified homotopy equivalences between them.  

Alternatively, we will construct a spectrum for each finite-dimensional subspace of $H$ satisfying certain conditions and will show the resulting spectra are all equivalent. An important requirement for such a subspace is ensure that the hypothesis for Proposition~\ref{prop suspindex} holds. The next proposition gives quantitative criteria for such suitable subspaces.

\begin{proposition} There exist sufficiently small positive real numbers \(c_{1}, c_2\) so that if a finite-dimensional subspace \(V\) satisfies
\begin{enumerate}
\item \(V\) contains \(\ker{L}\), 
\item $ \|\pi_V L - L \pi_V \| \leq c_1  $,
\item \(\sup_{x \in X} \left\| (1 - \pi_V)Qx  \right\| \leq c_2\).

\end{enumerate}
Then, the Conley index $\cindex{ X \cap V , \eta_V}$ is defined. In addition, if \(V\) and \(W\) are two such spaces and \(W = V \oplus U\), then  \(\pi_U L_{|U}\) has no kernel and \(X\) is an isolating neighborhood for the family of flows generated by \(s F_{V,W} + (1-s)F_W\).
\label{prop indep}
\end{proposition}

\begin{proof} By Proposition~\ref{prop hilbertpseudo}, there is $\epsilon_{0} > 0$ such that $X$ is an isolating neighborhood for any flow generated by $F' = L' + Q'$ with $\bar{\rho}_{X} (L' + Q', L + Q) < \epsilon_{0}$. We want the vector fields \(F_V , F_W \) and \( F_{V,W}\) (see (\ref{eq F_V}) and (\ref{eq F_V,W})) to get close to \(L+Q\), so we will require \(c_2 < \epsilon_0 /2 \) and \(c_1 < \epsilon_0 /4 \) the estimate from Lemma~\ref{lemma F_V,W}.

Next, we would also want that \(\pi_U L_{|U}\) has no kernel. Since $L$ is Fredholm, there is $\delta_0 > 0$ such that $(- \delta_0 , \delta_0 ) \cap \Spec{L} \subset\{0\} $.  For $x \in U$, we have an identity 
\begin{align} \label{identityu1}
        \pi_U L x =  L x + (\pi_U L - L \pi_U) x = L x + (\pi_W L - L \pi_W )x  - (\pi_V L - L \pi_V )x.
\end{align}
Thus, if we require that \(c_1 < \delta_0 /2\), we will have \(\|\pi_U L x \| > 0\) for all nonzero \(x \in U\) as \(U \subset V^\bot \subset (\ker{L})^{\bot}\).

\end{proof}
 
We will denote a collection of subspaces satisfying hypotheses of Proposition~\ref{prop indep} by \(\mathcal{V}_{L,Q}\) without mentioning the constants. To investigate some features of \(\mathcal{V}_{L,Q}\), we begin with the following fact. 

\begin{lemma} Suppose that \(L\) admits an \exhaust sequence. For any finite dimensional subspace $W$ and positive number $\epsilon$, one can find a finite dimensional subspace $E$ such that $W \subset E$ and $\| L \pi_E - \pi_E L \| < \epsilon$.
\label{lemma cofinal}
\end{lemma} 
\begin{proof} Given $W$ and $\epsilon_1 , \epsilon_2 > 0$, we can pick a subspace $E_n$ from an \exhaust sequence with sufficiently large $n$ so that
\[  \sup_{x \in B(W,1)} \left\| (1 - \pi_{E_n})x  \right\| < \epsilon_1 \quad \text{and} \quad \| L \pi_{E_n} - \pi_{E_n} L \| < \epsilon_2 . \]
The first inequality says that the image $\pi_{E_n}(W)$ is sufficiently close to $W$. We can decompose $E_n = \pi_{E_n}(W) \oplus U$ and take $E$ to be the subspace spanned by $W$ and $U$. It is not hard to check that $\| \pi_E - \pi_{E_n} \| < \epsilon_1$. When $\epsilon_1 , \epsilon_2$ are sufficiently small, we have $ \| L \pi_E - \pi_E L \| < \epsilon_2 + \epsilon_1 \|L \| < \epsilon $.

\end{proof}

For a subspace $V$ in  \(\mathcal{V}_{L,Q}\), we define a spectrum \(E(X,F,L,V)\) by 
\begin{align*} E(X,F,L,V) := \Sigma^{-V^-} \cindex{ X \cap V , \eta_V}
 \end{align*} 
where we apply the spectrification functor to \(\cindex{ X \cap V , \eta_V}\) and desuspend by \(V^-\), the negative eigenspace with respect to \(\pi_V L_{|V}\). This is well-defined up to canonical homotopy equivalence when concerning the choice of index pairs. The idea of desuspending the Conley index by the negative eigenspace was used by Manolescu in \cite{Man1}. 

We now prove invariance of $E(X,F,L,V)$ with respect to the choice of \(V \in \mathcal{V}_{L,Q}\).

\begin{corollary} Suppose that $L$ admits an \exhaust sequence. Then $E(X,F,L,V_1) \cong E(X,F,L,V_2)$ when \(V_1 , V_2 \in  \mathcal{V}_{L,Q} \).
\label{cor invariance}
\end{corollary}
\begin{proof}
We first consider the case \(V , W \in  \mathcal{V}_{L,Q} \) with \(W = V \oplus U\). This is the same setup as in the hypothesis of Proposition~\ref{prop suspindex}. Hence, we see that \(E(X,F,L,V)\) and \(E(X,F,L,W)\) are equivalent by suspending both with \(S^W\) and using the fact that
\[ V^+ \oplus U \cong W^+ \oplus U^- . \]

For general \(V_1 , V_2 \in \mathcal{V}_{L,Q}\), we only need to find a subspace \(W \in \mathcal{V}_L\) which contains both \(V_1\) and \(V_2\). By Lemma~\ref{lemma cofinal}, we can find a subspace \(W\) that contains the span of \(V_1\) and \(V_{2}\) and satisfies $ \|\pi_W L - L \pi_W \| \leq c_1  $. It is easy to check that  \(W \in \mathcal{V}_{L,Q}\) and the proof is complete.  

Note that Lemma~\ref{lemma cofinal} also ensures that \(\mathcal{V}_{L,Q}\) is nonempty. Moreover, we just showed that \(\mathcal{V}_{L,Q}\) is cofinal under inclusion of subspaces. 
\end{proof}

\begin{remark} At this point, we need to assume that $L$ admits an \exhaust sequence. Without Lemma~\ref{lemma cofinal}, we cannot guarantee that two Conley indices are equivalent. 
\end{remark}

At this point, we may consider $E(X,F,L,V)$ without referring to a subspace \(V\), but it might still depend on the choice of permissible decompositions. For a different \goodvf decomposition \(F = L' + Q'\), it is not hard to see that \(\mathcal{V}_{L,Q} \cap \mathcal{V}_{L',Q'}\) is nonempty.
This follows from the fact that \(L'-L\) is compact and that we have an identity (cf. (\ref{eq sqrt2LV}))
\begin{align*}
 \left\| \pi_V (L' -L) - (L' - L) \pi_V \right\| 
\leq  \sqrt{2} \left\| (1-\pi_V) (L' -L) \right\|.
\end{align*}

As discussed earlier, the Conley index $\cindex{ X \cap V , \eta_V}$ itself does not depend on the choice of \(L\) or the choice of decompositions,  but it is only the negative subspace \(V^-\) used in desuspension that depends on \(L\). For \(V \in \mathcal{V}_{L,Q} \cap \mathcal{V}_{L',Q'} \), we see that the spectra \(E(X,F,L,V)\) and \(E(X,F,L',V)\) differ only by  a sphere suspension of  certain dimension. One may view the choice of \(L\) involved here as a grading for an object \(E(X,F)\).
 
\begin{example} We now consider a linear flow in infinite-dimensional setting. Let $L$ be a self-adjoint operator and $\eta$ be the flow generated by $L$. Suppose that $L$ has no kernel and its spectrum is simple so that $L$ admits an \exhaust sequence and there is a decomposition $H = \bigoplus V_{\lambda}$ as a sum of 1-dimensional eigenspaces.

Like the finite-dimensional analog, the origin is an isolated invariant set for $\eta$ and we can pick a corresponding isolating neighborhood $X$ to be the unit ball in $H$. Note that this ball is closed and bounded but infinite-dimensional. 

Let $V = \sum_{i=1}^n V_{\lambda_i}$ be a finite sum of eigenspaces. Clearly, $V$ satisfies the hypothesis of Proposition~\ref{prop indep}, so that we can consider a compressed flow on $V$ and its Conley index. From Example~\ref{ex linearflow}, the Conley index $\cindex{ X \cap V , \eta_V}$ is has a homotopy type of $S^{n^-}$, where $n^-$ is a number of negative eigenvalues of $L$ restricted to $V$. Therefore, 
$$E(X,L,L) \cong S^0 .$$ 

On the other hand, we can consider a perturbation of $L$ by a finite-rank operator that switches the sign of an eigenvalue. More specifically, suppose that $\lambda_1$ is negative and that $L'$ is an operator given by $L' = -L$ on $V_{\lambda_1}$ and $L' = L$ on its orthogonal complement. Then $K = L - L'$ is finite-rank and $L = L' + K$ is another \goodvf decomposition of $L$. However, the number of negative eigenvalues of $L'$ in $V$ is now $n^- - 1$, so instead 
$$E(X,L,L') \cong S^1 . $$ 
  
\end{example}

Lastly, we will consider a \goodvf vector field near \(F\), that is a vector field $F' = L' + Q'$ such that $\bar{\rho}_{X} (L' + Q', L + Q)$ is sufficiently small. It follows that \(X\) is an isolating neighborhood for \(F'\) and we will be able to use a subspace \(V \in \mathcal{V}_{L,Q}\) to construct its stable Conley index. 

Let us consider difference between \(F\) and a compression of \(F'\) on $V$ using the pseudometric $\bar{\rho}_X$. The linear part term has the following estimate
 \begin{align*} & \|L -  \pi_V L' \pi_V - (1-\pi_V) L' (1-\pi_V )\|  \\
  &\leq  \left\| \pi_V L - L \pi_V \right\|  + \left\| \pi_V (L-L') \pi_V + (1-\pi_V) (L-L') (1-\pi_V )\right\| \\
  &\leq \left\| \pi_V L - L \pi_V \right\| + 2\left\|L' -L  \right\| .    \end{align*}
For the compact part term, we have
\begin{align*}\sup_{x \in X} \left\| Q - \pi_VQ'x  \right\| &\leq \sup_{x \in X} \left\| (1 - \pi_V)Qx  \right\| + \sup_{x \in X} \left\| \pi_V (Q-Q')x  \right\| \\ 
&\leq \sup_{x \in X} \left\| (1 - \pi_V)Qx  \right\| + \sup_{x \in X} \left\|  (Q-Q')x  \right\|.  \end{align*}
Thus, when the terms $\bar{\rho}_{X} (L' + Q', L + Q)$, $\|[L,\pi_V ]\|$, and $\sup_{x \in X} \left\| (1 - \pi_V)Qx  \right\|$ are sufficiently small, $X$ is an isolating neighborhood for the compression of $F'$ on $V$ and we obtain the Conley index $\cindex{ X \cap V , {\eta'}_V}$ for $V \in \mathcal{V}_{L,Q}$. It is also clear that $X$ is an isolating neighborhood for the family $s F_V + (1-s) F'_V$ parametrized by the interval, so that $\cindex{ X \cap V , \eta'_V} \cong \cindex{ X \cap V , \eta_V}$ by continuation of the Conley index. When $\left\|L' -L  \right\|$ is sufficiently small, the signature of \(\pi_V L'_{|V}\) is the same as \(\pi_V L_{|V}\). Therefore,
$$E(X,F' , L' ) \cong E(X,F,L)$$

When $L'$ also admits an \exhaust sequence, we can show that $\mathcal{V}_{L' ,Q'}$ and $\mathcal{V}_{L,Q}$ have nonempty intersection. For instance, we can choose $V$ that contains both \(\ker{L}\) and \(\ker{L'}\). Then, we can choose $V,W \in \mathcal{V}_{L,Q}$ so that \(X\) is an isolating neighborhood for the family of flows generated by \(s F'_{V,W} + (1-s)F'_W\). Finally, we observe that
\begin{align*}
        \pi_U L' x &= \pi_U (L' - L) x  + \pi_U L x \\
        &= \pi_U (L' - L) x  + L x + (\pi_W L - L \pi_W ) x - (\pi_V L - L \pi_V ) x ,
\end{align*}
so that \(\pi_V L'_{|U}\) is nondegenerate when the terms \(\|L'-L\| ,  \|[L,\pi_V ]\| \), and \( \|[L,\pi_W ]\|\) are sufficiently small. Hence, we can conclude that the intersection of  $\mathcal{V}_{L' ,Q'}$ and $\mathcal{V}_{L,Q}$ is nonempty and have the following conclusion.

\begin{corollary} $E(X,F)$ is invariant under continuation of vector fields.
\end{corollary}
  


At the end, we mention another possible formulation for the stable Conley index. Observe the relation (\ref{eq indexsusp}) again
 \begin{align*} \cindex{X \cap W , \eta_W} \cong \cindex{ X \cap V , \eta_V} \wedge S^{U^-} . \end{align*}
We see that the Conley index does not change if we add $V$ by any positive definite subspace. In other words, one could define a Conley index of a compressed flow on a semi-infinite dimensional subspace.

A slightly different result of the Conley index on an infinite-dimensional space was developed by Rybakowski \cite{Rybakowski}. One of the necessary hypotheses for this setup can be roughly viewed as compactness of the exit set of an isololating neighborhood. For such an isolating neighborhood, it turns out that the Conley index can be defined and behaved in essentially the same way as in the finite-dimensional case. 

Let $H^+$ be a positive eigenspace with respect to $L$ and $F = L + Q$ be a permissible vector field. When $X \cap H^+$ is an isolating neighborhood with respect to a flow generated by $\pi_{H^+} F$ restricted to $H^+$, we can show that $X \cap H^+$ satisfies Rybakowski's hypotheses by using the same argument as in Proposition~\ref{properh}. This way, we can obtain the Conley index $\cindex{X \cap H^+ , \eta_{H^+}}$. 

The collection of relevant subspaces can be characterized as a positive eigenspace of $L'$ for any operator differed from $L$ by a compact operator. This suggests that we look at a semi-infinite subspace that lies in the same polarization class as $H^+ \oplus H^- $. Consequently, we could form a semi-infinite spectrum (cf. \cite{Doug}). However, we cannot guarantee that $X \cap H^+$ is an isolating neighborhood for all such $H^+$, so it might be difficult to define a proper object in a stable category this way. Moreover, the development of theory of semi-infinite spectra is also at an early stage. 



\bibliography{research}

\begin{thebibliography}{10}

\bibitem{Conley}
C.~Conley.
\newblock {\em Isolated invariant sets and the {M}orse index}, volume~38 of
  {\em CBMS Regional Conference Series in Mathematics}.
\newblock American Mathematical Society, Providence, R.I., 1978.

\bibitem{Doug}
C.~L. Douglas.
\newblock Twisted parametrized stable homotopy theory, 2005.
\newblock eprint, arXiv:math/0508070.

\bibitem{Geba}
K.~G\c{e}ba, M.~Izydorek, and A.~Pruszko.
\newblock The {C}onley index in {H}ilbert spaces and its applications.
\newblock {\em Studia Math.}, 134(3):217--233, 1999.

\bibitem{TK1}
T.~Khandhawit.
\newblock {\em Twisted {M}anolescu-{F}loer spectra for {S}eiberg-{W}itten}.
\newblock PhD thesis, Massachusetts Institute of Technology, 2013.

\bibitem{ManK}
P.~Kronheimer and C.~Manolescu.
\newblock Periodic {F}loer pro-spectra from the {S}eiberg-{W}itten equations,
  2002.
\newblock eprint, arXiv:math/0203243.

\bibitem{Man1}
C.~Manolescu.
\newblock Seiberg-{W}itten-{F}loer stable homotopy type of three-manifolds with
  {$b_1=0$}.
\newblock {\em Geom. Topol.}, 7:889--932 (electronic), 2003.

\bibitem{Man13-1}
C.~Manolescu.
\newblock Pin(2)-equivariant {S}eiberg-{W}itten {F}loer homology and the
  triangulation conjecture, 2013.
\newblock eprint, arXiv:1303.2354.

\bibitem{May2}
J.~P. May.
\newblock {\em Equivariant homotopy and cohomology theory}, volume~91 of {\em
  CBMS Regional Conference Series in Mathematics}.
\newblock Published for the Conference Board of the Mathematical Sciences,
  Washington, DC, 1996.
\newblock With contributions by M. Cole, G. Comeza{\~n}a, S. Costenoble, A. D.
  Elmendorf, J. P. C. Greenlees, L. G. Lewis, Jr., R. J. Piacenza, G.
  Triantafillou, and S. Waner.

\bibitem{Rab}
P.~H. Rabinowitz.
\newblock {\em Minimax methods in critical point theory with applications to
  differential equations}, volume~65 of {\em CBMS Regional Conference Series in
  Mathematics}.
\newblock Published for the Conference Board of the Mathematical Sciences,
  Washington, DC, 1986.

\bibitem{Rybakowski}
Krzysztof~P. Rybakowski.
\newblock On the homotopy index for infinite-dimensional semiflows.
\newblock {\em Trans. Amer. Math. Soc.}, 269(2):351--382, 1982.

\bibitem{SalaCon}
D.~Salamon.
\newblock Connected simple systems and the {C}onley index of isolated invariant
  sets.
\newblock {\em Trans. Amer. Math. Soc.}, 291(1):1--41, 1985.

\bibitem{TW}
C.~Troestler and M.~Willem.
\newblock Nontrivial solution of a semilinear {S}chr\"odinger equation.
\newblock {\em Comm. Partial Differential Equations}, 21(9-10):1431--1449,
  1996.

\end{thebibliography}
\bibliographystyle{plain}


\end{document}